\documentclass[11pt]{amsart}
\usepackage[utf8]{inputenc}
\usepackage{fullpage,url,amssymb,enumitem,colonequals, mathrsfs,comment}
\usepackage{microtype}
\usepackage{hyperref}
\usepackage{xcolor}
\usepackage{tikz}

\newcommand{\C}{\mathbb{C}}

\newcommand{\PP}{\mathbb{P}}

\newcommand{\Q}{\mathbb{Q}}

\newcommand{\Z}{\mathbb{Z}}

\newcommand{\kbar}{{\overline{k}}}

\newcommand{\calA}{\mathcal{A}}

\newcommand{\calS}{\mathcal{S}}

\DeclareMathOperator{\Gal}{Gal}

\DeclareMathOperator{\im}{im}

\DeclareMathOperator{\Spec}{Spec}

\newcommand{\Span}{\operatorname{Span}}

\newcommand{\et}{{\operatorname{et}}}

\newcommand{\GL}{\operatorname{GL}}

\newcommand{\Sp}{{\operatorname{Sp}}}

\newtheorem{theorem}{Theorem}[section]
\newtheorem{lemma}[theorem]{Lemma}

\theoremstyle{definition}

\theoremstyle{remark}
\newtheorem{remark}[theorem]{Remark}

\usepackage{microtype}
\title{Galois specialization to symmetric points and the inverse Galois problem up to $S_n$}

\author{Borys Kadets}
\address{Borys Kadets, Department of Mathematics\\Boyd Graduate Studies Research Center\\
	University of Georgia\\
	Athens, GA 30602}
\email{kadets.math@gmail.com}
\urladdr{\url{http://bkadets.github.io}}

\thanks{}
\date{\today}

\begin{document}
\begin{abstract}
	The paper is concerned with the following version of Hilbert's irreducibility theorem: if $\pi: X \to Y$ is a Galois $G$-covering of varieties over a number field $k$ and $H \subset G$ is a subgroup, then for all sufficiently large and sufficiently divisible $n$ there exist a degree $n$ closed point $y \in |Y|$ and $x\in \pi^{-1}(y)$ for which $k(x)/k(y)$ is a Galois $H$-extension, and $k(y)/k$ is an $S_n$-extension. The result has interesting corollaries when applied to moduli spaces of various kinds. For instance, for every finite group $G$ there is a constant $N$ such that for all $n>N$ there is a degree $n$, $S_n$-extension $F/\Q$ such that over $F$ the inverse Galois problem for $G$ has a solution.
	
\end{abstract}
\maketitle

\section{Introduction}

Suppose $X$ is a connected scheme, $\bar{y}$ is a geometric point of $X$, and $\pi_G:X_G \to X$ is a finite \'etale Galois $G$-covering. Such a covering corresponds to a morphism  $\pi_1^\et(X, \bar{y}) \twoheadrightarrow G$. For every closed point $x \in |X|$ and a geometric point $\bar{x}$ above $x$, choose a path from $\bar{x}$ to $\bar{y}$ and consider the composition $\phi_x: \pi_1(\Spec k(x), \bar{x}) \to \pi_1(X, \bar{x}) \xrightarrow{\sim} \pi_1(X, \bar{y}) \twoheadrightarrow G$. The image $H_x = \im \phi_x$ of this map describes the Galois action on the fiber above $x$; changing the path from $\bar{x}$ to $\bar{y}$ changes $H_x$ by a conjugation, and so every closed point $x \in |X|$ defines a conjugacy class of a subgroup $H_x \subset G$. Standard results in number theory, such as Chebotarev's density and Hilbert's irreducibility theorems, concern the distribution of the conjugacy classes $H_x$ as $x \in |X|$ varies. In this paper we consider the case when $X$ is an arbitrary smooth variety over a number field $k$ (a variety is a separated scheme of finite type over a field). In this setting the set of closed points $|X|$ is too complicated to talk about distribution questions in the analytic sense; the main result of this note is that every conjugacy class appears above some point $x \in |X|$, and that, moreover, we can require $k(x)/k$ to be a degree $n$, $S_n$-extension for all  sufficiently large and divisible $n$.  

Recall that the \emph{index} $i(X)$ of a variety $X/k$ is the greatest common divisor of the degrees $\deg P$ of closed points $P \in |X|$. A degree $n$ separable field extension $K/k$ is called an $S_n$-extension if the Galois group of the Galois closure of $K/k$ is the symmetric group $S_n$; we call degree $n$ closed point $P$ on a variety $X/k$ an \emph{$S_n$-point} if $k(x)/k$ is an $S_n$-extension.

\begin{theorem}\label{main theorem}
	Suppose $k$ is a number field, and $X/k$ is a smooth quasi-projective variety of dimension at least $1$. Suppose $X_G \to X$ is a finite \'etale Galois covering with Galois group $G$ such that $X_G$ is geometrically irreducible. Fix a subgroup $H \subset G$. Then there exists a constant $N$ such that for any finite extension $L/k$ and for any $n>N$ which is divisible by $i(X)[G:H]$, there exist infinitely many degree $n$ $S_n$-points $x \in |X_L|$ such that $H_x$ is conjugate to $H$. 
\end{theorem}
\begin{remark}
	In this theorem we may replace $X$ with an open subset that still contains a degree $d$ point. In particular, one can replace the quasi-projective assumption with the notion of ``FA-scheme'' in the sense of \cite[Section 2.2]{gabber2013index}.
\end{remark}
\begin{remark}
	Theorem \ref{main theorem} is related to a result of Poonen \cite[Theorem 1]{poonen2001points}, which implies that, in notation of our theorem, there is a closed point $x \in |X|$ with $H_x = \{e\}$ (but does not give control over the field extension $k(x)/k$.) 
\end{remark}

This result can be applied to a covering of moduli spaces to build objects with prescribed Galois structure. Taking $X$ to be the quotient $\PP^n_k/\!/G$ under a faithful action $G\curvearrowright \PP^n$ we obtain the following result (see Section \ref{applications}).

\begin{theorem}[Inverse Galois Problem up to $S_n$]\label{Inverse-Galois-application}
	Suppose $k$ is a number field, and $G$ is a finite group. Then there is a constant $N=N(G)$ such that for every $n>N$ there are infinitely many degree $n$ extensions $L/k$ with Galois group $S_n$, for which there exists an extension $F/L$ with Galois group $G$.  
\end{theorem}

Note that it is easy to realize $G$ over some extension of $k$, say, by embedding $G$ intro a symmetric group and taking the fixed field of $G$, but such constructions never realize $G$ over an $S_n$ extension. A weak version of Theorem \ref{Inverse-Galois-application} can be deduced from the main results of \cite{fried1992embedding}: from \cite[Corollary~1]{fried1992embedding}, one can realize any group $G$ over a $\prod_{n\geqslant 1} S_n$-extension of $k$.

The same idea can be used to construct abelian varieties with specified level $N$ structures.
\begin{theorem}\label{level-structures}
	Suppose $k$ is a number field that contains $N$-th roots of unity, $g\geqslant 1$ is an integer, and $H \subset \Sp(2g, \Z/N\Z)$ is a subgroup of index $d$. Then there exists a constant $M=M(H)$ such that if $n>M$ is an integer divisible by $d$, then there exists a degree $n$, $S_n$-extension $K/k$ and a $g$-dimensional abelian variety $A/K$ such that the image of the Galois action on the $N$-torsion points $\Gal_K \to \Sp(A[N])$ is conjugate to $H$.
\end{theorem}

\section{Main Theorem}
We first reduce the proof to the case of curves by combining Bertini and Lefschetz theorems.

\begin{lemma}\label{reduce-to-curves}
	In the setting of Theorem \ref{main theorem}, let $D \subset X$ be a closed subscheme which is a union of closed points $D=\bigcup_{i=1}^s P_i$ such that $\gcd(\deg P_i)=i(X)$. Then there exists a smooth geometrically integral curve $Z \subset X$ with $D \subset Z$ such that for a point $z \in (Z \setminus D)(\C)$ the natural map $\pi_1((Z\setminus D)(\C), z) \to \pi_1((X\setminus D)(\C), z)$ is surjective. 
\end{lemma}
\begin{proof}
	We prove the statement by induction on $\dim X$. The base case $\dim X = 1$ is immediate. If $\dim X>1$, we embed $X$ into a projective space $\PP^r$ and consider the intersection of $X$ with a general degree $n$ hypersurface $H \subset \PP^r$ passing through $D$. For $n$ large enough, this intersection $X'=X \cap H$ is smooth and geometrically irreducible by a suitable version of Bertini's theorem (see, for example, \cite[Theorems 1 and 7]{kleiman1979bertini}). The surjectivity on fundamental groups (for sufficiently large $n$) follows from a version of Lefschetz's theorem, as we now explain. After replacing $X \to \PP^r$ with a high degree Pl\"ucker embedding, we can assume that the span $\Span D$ of the points of $D$ intersects $X$ only at $D$, and that the projection $\pi_D$ from $\Span D$ has $\dim X$-dimensional image. In this language, the variety $X'\setminus D$ is a preimage of a hyperplane under $\pi_D$, and for $x \in X'\setminus D$ the surjectivity of $\pi_1((X'\setminus D)(\C),x) \to \pi_1((X\setminus D)(\C),x)$ follows from Lefschetz theorem in the form of \cite[Lemma 1.4]{deligne1981groupe}.
\end{proof} 
\begin{remark}
	Since the varieties are assumed to be smooth, if $\dim X > 1$ we have a natural isomorphism $\pi_1((X\setminus D)(\C),z) \xrightarrow{\sim} \pi_1(X(\C),z).$
\end{remark}

\begin{proof}[Proof of Theorem \ref{main theorem}]
Fix a collection of distinct closed points $P_1, \dots, P_s \in |X|$ such that $\gcd(\deg P_i)=i(X)$, and denote by $D$ the union $D=\bigcup_i P_i$. Let $Z$ be the smooth curve from Lemma \ref{reduce-to-curves}. Since the map $\pi_1((Z\setminus D)(\C), z) \to \pi_1((X\setminus D)(\C), z)$ is surjective, the covering $X_G \to X$ remains a geometrically connected Galois $G$-covering when pulled back to $Z$. Therefore it suffices to consider the case $\dim X =1$. In this case, after compactifying, we can assume that $X_G$ and $X$ are both smooth proper curves, $\pi_G: X_G \to X$ is a geometrically irreducible $G$-covering, and $D \subset X$ is a divisor which does not intersect the branch locus of $\pi_G$. 

Let $\pi_H:X_H \to X$ be the intermediate covering of $\pi_G$ corresponding to $H$, so that $\pi_{G/H}:X_G \to X_H$ is an $H$-covering. Another way of phrasing the theorem is that there are infinitely many degree $n$, $S_n$-points $x \in |X_L|$ and $L(x)$-rational points $x_H \in \pi_H^{-1}(x)$ such that $\pi_{G/H}^{-1}(x_H)$ is an irreducible scheme. Consider the collection $\calS$ of divisors of the form $E=m_1 \pi_H^{-1}(P_1) + \dots + m_s \pi_H^{-1}(P_s)$ for nonnegative integers $m_1, \dots, m_s$. Let $m$ be a constant satisfying the following conditions:
\begin{enumerate}
	\item $m > [G:H]=\deg \pi_H$;
	\item  $m>2 g(X_H)$, so that any divisor on $X_H$ of degree larger than $m$ is very ample;
	\item any integer larger than $m$ and divisible by $i(X)[G:H]$ is the degree of a divisor from $\calS$ ($m$ is larger than the Frobenius number of the semigroup of degrees of divisors from $\calS$).
\end{enumerate}
Note that $m$ can be chosen independently of $L.$ We claim that any $n>m$ and divisible by $i(X)[G:H]$ satisfies the conditions of the theorem. Fix an $n>m$ divisible by $i(X)[G:H]$ and choose a divisor $E \in \calS$ of degree $n$. Consider the embedding $X_H \subset \PP^r$ given by the complete linear system $|E|$. 
 
 To simplify notation, for the remainder of the proof all varieties are considered after a base change to $L$.
 Consider the correspondences $I_G$, $I_H$ that parameterize incidences between points $x$ on $X_G$ and $X_H$ and hyperplanes $h$ in $\PP^{|E|}$: $I_G \subset X_G \times \left(\PP^{|E|}\right)^\vee$ and $I_H \subset X_H \times \left(\PP^{|E|}\right)^\vee$, given by
 \[I_G=\{(x, h)\colon \pi_{G/H}(x) \in h\}\] and \[I_H=\{(x, h)\colon x \in h\}.\]
 The variety $I_G$ is irreducible, since the projection $I_G \to X_G$ is as a proper map with irreducible equidimensional fibers (the fibers are projecitve spaces of dimension $\dim |E|-1$). The covering $I_H \to (\PP^{|E|})^\vee$ has degree $n$ and monodromy group $S_n$ (see, for example,  \cite[Lemma, Chapter III, page 111]{ACGH}). Therefore, by Hilbert irreducibility theorem applied to the (factored) covering $I_G \to I_H \to \left(\PP^{|E|}\right)^\vee$, a general hyperplane section $x_H = h \cap X_H$ is a degree $n$, $S_n$-point on $X_H$, and moreover, since $I_G$ is irreducible, the fiber $\pi_{G/H}^{-1}(x_H)$ is irreducible as well (see \cite[Chapter 9]{serre1989lectures} for Hilbert's theorem in a geometric form). Finally, consider the image $x \colonequals \pi_H(x_H)$. The field extension $L(x_H)/L$ has no intermediate subextensions, and so either $L(x)=L$, or $L(x)=L(x_H)$. If $x$ is a rational point, then the degree of $x_H \subset \pi_H^{-1}(x)$ is at most $[G:H]$, contradicting the assumption $n\geqslant m > [G:H]$. Therefore $x$ is the degree $n$, $S_n$-point on $X$ we seek. 
\end{proof}
\section{Applications}\label{applications}
\begin{proof}[Proof of Theorem \ref{Inverse-Galois-application}]
	Consider an action of the group $G$ on a vector space $V$ over $k$, such that the induced action on $\PP(V)$ is faithful, and let $\pi$ be the associated quotient map $\pi: \PP(V) \to \PP(V)/\!/G$. Consider a smooth affine open subset $U \subset \PP(V)/\!/G$ over which $\pi$ is an \'etale Galois $G$-covering. Since the rational points are dense in $\PP(V)/\!/G$, the variety $U$ has rational points and so $i(U)=1$. Applying Theorem \ref{main theorem} to $X=U$, $X_G=\pi^{-1}(U)$, and $H=G$ gives the conclusion.
\end{proof}

\begin{proof}[Proof of Theorem \ref{level-structures}]
Consider any abelian scheme $\calA \to X$ over a smooth base $X$ equipped with a geometric point $\bar{x}$ above a rational point $x \in X(k)$ such that the action $\bar{\rho}_N\colon \pi_1^{\et}(X_{\kbar}, \bar{x}) \to \Sp(2g, \Z/N\Z)$ on the $N$-torsion is surjective. There are many different such families for which the monodromy is known to be surjective, for instance it is true for the Jacobian of the universal curve (see \cite[Section 5.12]{deligne1969irreducibility}). Since $k$ is assumed to have $N$-th roots of unity, the arithmetic monodromy -- image of $\rho_N:\pi_1^{\et}(X, x) \to \GL(A[N])$ -- is contained in $\Sp(A[N])$. Since $\bar{\rho}_N$ is surjective, so is the arithmetic monodromy $\rho_N$. Therefore the covering $X_N \to X$ corresponding to the kernel of $\rho_N$ is geometrically irreducible Galois \'etale $\Sp(2g, \Z/N\Z)$-covering of algebraic varieties. Applying Theorem \ref{main theorem} to this covering gives the result.  
\end{proof}
	\bibliographystyle{alpha}
	\bibliography{bibliography}
\end{document}